\documentclass[draft=false]{scrartcl}%

\usepackage{amsfonts}
\usepackage{amsmath}
\usepackage{amsthm}
\usepackage{cite}
\usepackage[T1]{fontenc}
\usepackage{hyperref}
\usepackage{lmodern}
\usepackage{mathscinet}
\usepackage{mathtools}
\usepackage{microtype}
\usepackage[automark]{scrpage2}%

\usepackage{151arxiv}

 \newtheorem{theorem}{Theorem}[section]
 \newtheorem{cor}[theorem]{Corollary}
 \newtheorem{lemma}[theorem]{Lemma}
 \newtheorem{proposition}[theorem]{Proposition}
\theoremstyle{definition}
\theoremstyle{remark}
 \newtheorem{remark}[theorem]{Remark}
 \newtheorem{exa}[theorem]{Example}

\mathtoolsset{mathic=true}

\numberwithin{equation}{section}

\title{Rational $q\times q$~\Car{} Functions and Central Non-negative \tH{} Measures}
\author{Bernd Fritzsche \and Bernd Kirstein \and Conrad M\"adler}

\begin{document}
\maketitle

\begin{abstract}
 We give an explicit representation of central measures corresponding to finite \tTnnd{} sequences of complex \tqqa{matrices}. Such measures are intimately connected to central \tqqa{\Car{}} functions. This enables us to prove an explicit representation of the non-stochastic spectral measure of an arbitrary multivariate autoregressive stationary sequence in terms of the covariance sequence.
\end{abstract}

\begin{description}
 \item[Mathematical Subject Classification (2000)] Primary 30E05, 60G10, Secondary 42A70
\end{description}

\section{Introduction}
 If \(\kappa\) is a \tnn{} integer or if \(\kappa=\infty\)\index{k@$\kappa$}, then a sequence \(\Cbk\) of complex \tqqa{matrices} is called \noti{\tTnnd{}} if, for each \tnn{} integer \(n\) with \(n\leq\kappa\), the block \Toe{} matrix \(\Tn\defeq\matauo{C_{j-k}}{j,k=0}{n}\)\index{t@$\Tn$} is \tnnH{}. In the second half of the 1980's, the first two authors intensively studied the structure of \sTnnd{} sequences of complex \tqqa{matrices} in connection with interpretations in the languages of stationary sequences, \Car{} interpolation, orthogonal matrix polynomials etc. (see~\cite{MR885621I_III_V,MR1056068} and also~\cite{MR1152328} for a systematic treatment of several aspects of the theory).

 In particular, it was shown in~\cite[Part~I]{MR885621I_III_V} (see also~\cite[\csec{3.4}]{MR1152328}) that the structure of the elements of a \tTnnd{} sequence of complex \tqqa{matrices} is described in terms of matrix balls which are determined by all preceding elements. Amongst these sequences there is a particular subclass which plays an important role, namely the so-called class of central \sTnnd{} sequences of complex \tqqa{matrices}. These sequences are characterized by the fact that starting with some index all further elements of the sequences coincide with the center of the matrix ball in question. Central \sTnnd{} sequences possess several interesting extremal properties (see~\cite[Parts~I--III]{MR885621I_III_V}) and a remarkable recurrent structure (see~\cite[\cthm{3.4.3}]{MR1152328}).

 In view of the matrix version of a classical theorem due to Herglotz (see, \eg{}~\cite[\cthm{2.2.1}]{MR1152328}), the set of all \sTnnd{} sequences coincides with the set of all sequences of \tFcc{\tqqa{\tnnH} Borel measures} on the unit circle \(\T\defeq\setaa{z\in\C}{\abs{z}=1}\)\index{t@$\T$} of \(\C\). If \(\Cbinf \) is a \sTnnd{} sequence of complex \tqqa{matrices} and if \(\mu\) denotes the unique \tqqa{\tnnH} Borel measure on \(\T\) with \(\Cbinf \) as its sequence of Fourier coefficients then we will call \(\mu\) the spectral measure\index{spectral measure} of \(\Cbinf \). In the special case of a central \sTpd{} sequence of complex \tqqa{matrices}, \ie{}, if for each \tnn{} integer \(n\) the block \Toe{} matrix \(\Tn\defeq\Cmsn\) is positive \tH{}, in~\cite[Part~III]{MR885621I_III_V} (see also~\cite[\csec{3.6}]{MR1152328}), we stated an explicit representation of its spectral measure. In particular, it turned out that in this special case its spectral measure is absolutely continuous with respect to the linear Lebesgue-Borel measure on the unit circle and that the corresponding Radon-Nikodym density can be expressed in terms of left or right orthogonal matrix polynomials.

 The starting point of this paper was the problem to determine the spectral measure of a central \sTnnd{} sequence of complex matrices. An important step on the way to the solution of this problem was gone in the paper~\cite{MR2104258}, where it was proved that the matrix-valued \Car{} function associated with a central \sTnnd{} sequence of complex matrices is rational and, additionally, concrete representations as quotient of two matrix polynomials were derived. Thus, the original problem can be solved if we will be able to find an explicit expression for the \tRHm{} of a rational matrix-valued \Car{} function. This question will be answered in \rthm{M3.2}. As a first essential consequence of this result we determine the \tRHm{}s of central matrix-valued \Car{} functions (see \rthm{C1}). Reformulating \rthm{C1} in terms of \sTnnd{} sequences, we get an explicit description of the spectral measure of central \sTnnd{} sequences of complex matrices (see \rthm{C2}).

 In the final \rsec{S1038}, we apply \rthm{C2} to the theory of multivariate stationary sequences. In particular, we will be able to express explicitly the non-stochastic spectral measure of a multivariate autoregressive stationary sequence by its covariance sequence (see \rthm{C}).

\section{On the \tRHm{} of rational matrix-valued \Car{} functions}\label{S1024}
 In this section, we give an explicit representation of the \tRHm{} of an arbitrary rational matrix-valued \Car{} function.
 
 Let \(\R\)\index{r@$\R$}, \symb{\Z}{z}, \symb{\NO}{n}, and \symb{\N}{n} be the set of all real numbers, the set of all integers, the set of all non-negative integers, and the set of all positive integers, respectively. Throughout this paper, let \(p,q\in\N\)\index{p@$p$}\index{q@$q$}. If \(\mathcal{X}\) is a non-empty set, then by \sym{\mathcal{X}^{q\times p}} we denote the set of all \tqpa{matrices} each entry of which belongs to \(\mathcal{X}.\) The notation \sym{\mathcal{X}^q} is short for \(\mathcal{X}^{q\times 1}\). If \(\mathcal{X}\) is a non-empty set and if \(x_1, x_2,\dotsc, x_q\in\mathcal{X}\), then let\index{c@$\col  (x_j)_{j=1}^q $} 
\[
\col  (x_j)_{j=1}^q 
\defeq  \begin{bmatrix} x_1\\ x_2 \\ \vdots \\ x_q \end{bmatrix}.
\]

 For every choice of \(\alpha, \beta,\in\R\cup \{-\infty,+\infty\}\), let \(\mn{\alpha}{\beta} \defeq  \setaa{m\in\Z}{\alpha \le m\le \beta}\)\index{z@$\mn{\alpha}{\beta}$}. We will use \symb{\Iq}{i} and \symb{\Oqp }{o} for the unit matrix belonging to \(\Cqq \) and the null matrix belonging to \(\Cqp \), respectively. For each \(A\in\Cqq\), let \(\re  A \defeq  \frac{1}{2} (A +A^\ad)\)\index{r@$\re  A$} and \(\im  A \defeq  \frac{1}{2\iu} (A-A^\ad)\)\index{i@$\im  A$} be the real part and the imaginary part of \(A\), respectively. If \(\kappa\in\NOinf \), then a sequence \(\Cbk \) of complex \tqqa{matrices} is called \noti{\tTnnd{}} (resp.\ \noti{\tTpd{}}) if, for each \(n\in\mn{0}{\kappa}\), the block \Toe{} matrix
\begin{align*}
\Tu{n}
\defeq\matauo{C_{j-k}}{j,k=0}{n}
\end{align*}
 \index{t@$\Tu{n}$}is \tnnH{} (resp.\ \tpH{}). Obviously, if \(m\in\NO \), then \(\Cb{m}\) is \tTnnd{} (resp.\ \tTpd{}) if the block \Toe{} matrix \(\Tu{m} =\matauo{C_{j-k}}{j,k=0}{m}\) is \tnnH{} (resp.\ \tpH{}).

 Let \(\Omega\) be a non-empty set and let \(\gA\) be a \(\sigma\)\nobreakdash-algebra on \(\Omega\). A mapping \(\mu\) whose domain is \(\gA\) and whose values belong to the set \symb{\Cggq}{c} of all \tnnH{} complex \tqqa{matrices} is said to be a \notion{\tnnH{} \tqqa{measure} on \((\Omega,\gA)\)}{measure!\tnnH{}} if it is countably additive, \ie{}, if \(\mu( \bigcup_{k=1}^\infty A_k) =\sum_{k=1}^\infty \mu (A_k)\) holds true  for each sequence \((A_k)_{k=1}^\infty\) of pairwise disjoint sets which belong to \(\gA\). The theory of integration with respect to \tnnH{} measures goes back to Kats~\cite{MR0080280} and Rosenberg~\cite{MR0163346}. In particular, we will turn our attention to the set \(\MggqT\) of all \tnnH{} \tqqa{measure}s on  \((\T , \BsaT  )\), where \(\BsaT\)\index{b@$\BsaT$} is the \(\sigma\)\nobreakdash-algebra of all Borel subsets of the unit circle \(\T\defeq\setaa{z\in\C}{\abs{z}=1}\)\index{t@$\T$} of \(\C\).

 Non-negative \tH{} measures belonging to \(\MggqT \) are intimately connected to the class \symb{\CqD}{c} of all \tqqa{\Car{}} functions in the open unit disk \(\D \defeq\setaa{z\in\C}{\abs{z}<1}\)\index{d@$\D$} of \(\C\). A \tqqa{matrix}-valued function \(\Phi \colon\D \to \Cqq \) which is holomorphic in \(\D \) and which fulfills \(\re  \Phi (z)\in\Cggq \) for all \(z\in\D\) is called \notion{\tqqa{\Car{}} function in \(\D \)}{\Car{} function}. The matricial version of a famous theorem due to F.~Riesz and G.~Herglotz illustrates the mentioned interrelation:

\begin{theorem}\label{D2.2.2}
\benui
 \item\label{D2.2.2.a} Let  \(\Phi\in\CqD \). Then there exists one and only one measure \(\mu\in\MggqT\) such that
\begin{equation}\label{NGZ}
\Phi (z) - \iu\im  \Phi (0) = \int_\T  \frac{\zeta + z}{\zeta - z} \mu (\dif\zeta)
\end{equation}
for each \(z\in \D \). For every choice of \(z\) in \(\D \), furthermore,
\[
\Phi (z) - \iu\im  \Phi (0) = \Fcm{0} + 2\sum_{j=1}^\infty \Fcm{j} z^j
\]
 where 
 \begin{equation} \label{D2.2.9}
  \Fcm{j}
  \defeq \int_\T  \zeta^{-j} \mu (\dif\zeta),
 \end{equation}
 \index{c@$\Fcm{j}$}for each \(j\in\Z\) are called the \notion{\tFcc{\(\mu\)}}{\tFc{}}.
 \item\label{D2.2.2.b} Let \(H\) be a \tH{} complex \tqqa{matrix} and let \(\mu\in\MggqT\). Then the function \(\Phi \colon\D \to\Cqq \) defined by
\[
\Phi (z) \defeq   \int_\T  \frac{\zeta + z}{\zeta - z} \mu (\dif\zeta) + \iu H
\]
belongs to \(\CqD \) and fulfills \(\im  \Phi (0) = H\).
\eenui
\end{theorem}

 A proof of \rthm{D2.2.2} is given, \eg{}, in~\cite[\cthm{2.2.2}, pp.~71/72]{MR1152328}. If \(\Phi \in\CqD \), then the unique measure \(\mu\in\MggqT\) which fulfills \eqref{NGZ} for each \(z\in\D \) is said to be the \notion{\tRHm{} of \(\Phi\)}{\tRHm{}}.

 Let \symb{\kron{u}}{d} be the Dirac measure on \((\T,\BsaT  )\) with unit mass at \(u\in\T\).
 
\begin{exa}\label{E1223}
 Let \(u\in\T\) and \(W\in\Cggq\). Then \rthm{D2.2.2} yields that the function \(\Phi\colon\D\to\Cqq\) defined by \(\Phi(z)\defeq\frac{u+z}{u-z}W\) belongs to \(\CqD\) with \tRHm{} \(\mu\defeq\kron{u}W\). The \tFcc{\(\mu\)} are given by \(\Fcm{j}=u^{-j}W\) for all \(j\in\Z\) and the function \(\Phi\) admits the representation \(\Phi(z)=[1+2\sum_{j=1}^\infty(zu)^j]W\) for all \(z\in\D\).
\end{exa}

Let \symb{\ran{A}}{r} and \symb{\nul{A}}{n} be the column space and the null space of a \tpqa{complex} matrix \(A\), respectively.

\bleml{L1103}
 Let \(\Phi\in\CqD\) with \tRHm{} \(\mu\). For all \(z\in\D\),
 \begin{align*}
  \Ran{\Phi(z)-\iu\im\Phi(0)}&=\Ran{\mu(\T)}=\Ran{\re\Phi(z)}
  \sand{}
  \Nul{\Phi(z)-\iu\im\Phi(0)}&=\Nul{\mu(\T)}=\Nul{\re\Phi(z)}.
 \end{align*}
\elem
\bproof
 Let \(z\in\D\). Since \(\re(\Phi(z)-\iu\im\Phi(0))=\re\Phi(z)\in\Cggq\), we obtain from~\zitaa{MR3014198}{\clem{A.8}, \cpartss{(a)}{(b)}} then \(\ran{\re\Phi(z)}\subseteq\ran{\Phi(z)-\iu\im\Phi(0)}\) and \(\nul{\Phi(z)-\iu\im\Phi(0)}\subseteq\nul{\re\Phi(z)}\). In view of \eqref{NGZ}, the application of~\zitaa{MR2988005}{\clem{B.2(b)}} yields \(\ran{\Phi(z)-\iu\im\Phi(0)}\subseteq\ran{\mu(\T)}\) and \(\nul{\mu(\T)}\subseteq\nul{\Phi(z)-\iu\im\Phi(0)}\). From \eqref{NGZ}, we get \(\re\Phi(z)=\int_\T(1-\abs{z}^2)/\abs{\zeta-z}^2\mu(\dif\zeta)\). Since \((1-\abs{z}^2)/\abs{\zeta-z}^2>0\) for all \(\zeta\in\T\), the application of~\zitaa{MR2988005}{\clem{B.2(b)}} yields \(\ran{\re\Phi(z)}=\ran{\mu(\T)}\) and \(\nul{\re\Phi(z)}=\nul{\mu(\T)}\), which completes the proof.
\eproof

 Now we consider the \tRHm{}s for a particular subclass of \(\CqD\). In particular, we will see that in this case, the \tRHm{} is absolutely continuous with respect to the \noti{linear Lebesgue measure} \(\lebc \)\index{l@$\lebc $} defined on \(\BsaT\) and that the Radon-Nikodym density can be always chosen as a continuous function on \(\T\).

 By a region\index{region} of \(\C\) we mean an open, connected, non-empty subset of \(\C\). For all \(z\in\C\) and all \(r\in(0,+\infty)\), let \(\diskaa{z}{r}\defeq\setaa{w\in\C}{\abs{w-z}<r}\)\index{k@$\diskaa{z}{r}$}.

\begin{lemma}\label{M1.2}
 Let \(\cD\) be a region of \(\C\) such that \(\cdisk{r}\subseteq\cD\) for some \(r\in(1,+\infty)\) and let \(F \colon\cD\to \Cqq \) be holomorphic in \(\cD\) such that the restriction \(\Phi\) of \(F\) onto \(\D\) belongs to \(\CqD \). Then the \tRHm{} \(\mu\) of \(\Phi\) admits the representation
\[%
 \mu (B)
 = \frac{1}{2\pi} \int_B \re F(\zeta) \lebca{\dif\zeta},
\]
 for each \(B\in\BsaT  \).
\end{lemma}
 A proof of \rlem{M1.2} can be given by use of a matrix version of an integral formula due to H.~A.~Schwarz (see, \eg{}~\cite[p.~71]{MR1152328}).

 In particular, \rlem{M1.2} contains full information on the \tRHm{s} of that functions belonging to \(\CqD\) which are restrictions onto \(\D\) of rational matrix-valued functions without poles on \(\T\). Our next goal is to determine the \tRHm{} of functions belonging to \(\CqD\) which are restrictions onto \(\D\) of rational matrix-valued functions having poles on \(\T\). First we are going to verify that in this case all poles on \(\T\) have order one. Our strategy of proving this is based on the following fact:
\begin{lemma}\label{M1.1}
 Let \(\Phi \in\CqD \) with \tRHm{} \(\mu\). For each \(u\in\T \), then
\begin{equation}\label{MB1}
 \mu\rk*{\{u\}}
 =\lim_{r\to 1 - 0} \frac{1-r}{2} \Phi (ru).
\end{equation}
\end{lemma}
 A proof of \rlem{M1.1} is given, \eg{}, in~\cite[\clem{8.1}]{MR1004239}. As a direct consequence of \rlem{M1.1} we obtain:
\breml{R1224}
 Let \(\cD\) be a region of \(\C\) such that \(\cdisk{r}\subseteq\cD\) for some \(r\in(1,+\infty)\) and let \(F\colon\cD\to\Cqq\) be holomorphic such that the restriction \(\Phi\) of \(F\) onto \(\D\) belongs to \(\CqD\). Then the \tRHm{} \(\mu\) of \(\Phi\) fulfills \(\mu(\set{u})=\Oqq\) for all \(u\in\T\).
\erem

\bpropl{P1309}
 Let \(\cD\) be a region of \(\C\) such that \(\cdisk{r}\subseteq\cD\) for some \(r\in(1,+\infty)\) and let \(F\) be a \tqqa{matrix-valued} function meromorphic in \(\cD\) such that the restriction \(\Phi\) of \(F\) onto \(\D\) belongs to \(\CqD\). Furthermore, let \(u\in\T\) be a pole of \(F\). Then \(u\) is a simple pole of \(F\) with \(\Res(F,u)=-2u\mu(\set{u})\) and
 \begin{equation}\label{P1309.B1}
  \lim_{r\to 1 - 0} \ek*{ (ru - u)F(ru)}
  =-2u\mu(\set{u}),
 \end{equation}
 where \symb{\Res(F,u)} is the residue of \(F\) at \(u\) and \(\mu\) is the \tRHm{} of \(\Phi\).
\eprop
\bproof
 Because of \rlem{M1.1}, we have \eqref{MB1}, which implies \eqref{P1309.B1}. Denote by \(k\) the order of the pole \(u\) of \(F\). Then \(k\in\N\) and
 \begin{equation}\label{P1309.1}
  \lim_{z\to u}(z-u)^kF(z)
  =A
  \neq\Oqq.
 \end{equation}
 In the case \(k>1\), we infer from \eqref{P1309.B1} that
 \[%
  \lim_{r\to 1 - 0} \ek*{ (ru - u)^kF(ru)}
  =\ek*{\lim_{r\to 1 - 0}(ru - u)^{k-1}}\ek*{\lim_{r\to 1 - 0} \ek*{ (ru - u)F(ru)}}
  =\Oqq,
 \]%
 which contradicts \eqref{P1309.1}. Thus \(k=1\) and the application of \eqref{P1309.B1} completes the proof.
\eproof

Since every complex-valued function \(f\) meromorphic in a region \(\cD\) of \(\C\) can be written as \(f=g/h\) with holomorphic functions \(g,h\colon\cD\to\C\), where \(h\) does not vanish identically in \(\cD\) (see, \eg{},~\zitaa{MR555733}{\cthm{11.46}}), we obtain:
\breml{R1336}
 For every \tpqa{matrix-valued} function \(F\) meromorphic in a region \(\cD\) of \(\C\), there exist a holomorphic matrix-valued function \(G\colon\cD\to\Cpq\) and a holomorphic function \(h\colon\cD\to\C\) which does not vanish identically in \(\cD\), such that \(F=h^\inv G\).
\erem

 If \(f\) is holomorphic at a point \(z_0\in\C\), then, for each \(m\in\NO \), we write \(f^{(m)} (z_0)\)\index{$f^{(m)}(z_0)$} for the \(m\)th derivative of \(f\) at \(z_0\).
\begin{lemma}\label{L1344}
 Let \(F\) be a \tpqa{matrix-valued} function meromorphic in a region \(\cD\) of \(\C\). In view of \rrem{R1336}, let \(G\colon\cD\to\Cpq\) and \(h\colon\cD\to\C\) be holomorphic such that \(h\) does not vanish identically in \(\cD\) and that \(F=h^\inv G\) holds true. Suppose that \(w\in\cD\) is a zero of \(h\) with multiplicity \(m>0\). Then \(w\) is a pole (including a removable singularity) of \(F\), the order \(k\) of the pole \(w\) fulfills \(0\leq k\leq m\), and \(h^{(m)} (w) \ne 0\) holds true. For all \(\ell\in\Z_{k, m}\), furthermore,
\begin{equation}\label{L1344.B}
 \lim_{z\to w} \ek*{  (z-w)^\ell F (z)}
 = \frac{m!}{(m-\ell)! h^{(m)} (w)} G^{(m-\ell)} (w).
\end{equation}
\end{lemma}
\begin{proof}
 Obviously \(w\) is a pole (or a removable singularity) of \(F\) and \(k\) fulfills \(0\leq k\leq m\). Since \(h\) is holomorphic, there is an \(r\in(0,+\infty)\) such that \(K\defeq\diskaa{w}{r}\) is a subset of \(\cD\) and \(h (z) \ne 0\) for all \(z\in K\setminus\set{w}\). Then \(F\) is holomorphic in \(K \setminus \{w\}\). Let \(\ell\in\Z_{k,m}\). Then there is a holomorphic function \(\Phi_\ell\colon K \to\Cpq \) such that \(F (z) = (z-w)^{-\ell}\Phi_\ell(z)\) for all \(z\in K \setminus \{w\}\). Consequently,
\begin{equation}\label{L1344.1}
 \lim_{z\to w} \ek*{  (z-w)^\ell F (z)}
 =\Phi_\ell(w).
\end{equation}
Since \(w\) is a zero of \(h\) with multiplicity \(m\ge \ell\), there exists a holomorphic function \(\eta_\ell\colon\cD\to\C\) such that \(h (z) = (z-w)^\ell \eta_\ell(z)\) holds true for all \(z\in\cD\). Furthermore, we have
\[
 h(z)
 =\sum_{j=m}^\infty\frac{h^{(j)} (w)}{j!} (z-w)^j
\]
for all \(z\in K\), where \(h^{(m)} (w) \ne 0\). Thus, for all \(z\in K\), we conclude
\[
 \eta_\ell(z)
 = \sum_{j=m}^\infty \frac{h^{(j)} (w)}{j!}  (z - w)^{j-\ell}.
\]
Comparing the last equation with the Taylor series representation of \(\eta_\ell\) centered at \(w\), we obtain \(\eta_\ell^{(s)} (w) = 0\) for all \(s\in\Z_{0, m-\ell-1}\) and
\[
 \frac{\eta_\ell^{(m-\ell) }(w)}{(m-\ell)!}
 = \frac{h^{(m)} (w)}{m!}.
\]
Using the general Leibniz rule for differentiation of products, we get then
\[
 (\eta_\ell\Phi_\ell)^{(m-\ell)} (w)
 = \sum_{s=0}^{m-\ell}\binom{m - \ell}{s}\ek*{\eta_\ell^{(s)} (w)}\ek*{\Phi_\ell^{(m-\ell-s)} (w)}
 = \frac{(m-\ell)!h^{(m)} (w)}{m!}\Phi_\ell(w),
\]
which, in view of \(h^{(m)} (w)\ne 0\), implies
\begin{equation}\label{L1344.2}
 \Phi_\ell(w)
 = \frac{m!}{(m-\ell)!h^{(m)} (w)}  (\eta_\ell\Phi_\ell)^{(m-\ell)} (w). 
\end{equation}
Obviously, we have
\[
 \eta_\ell(z)\Phi_\ell(z)
 =\eta_\ell(z) \ek*{ (z-w)^\ell F (z)}
 =h (z)F (z)
 =G(z)
\]
for all \(z\in K  \setminus \{w\}\). Since \(G\) is holomorphic, by continuity, this implies \((\eta_\ell\Phi_\ell) (z) = G (z)\)  for all \(z\in K \) and, hence \((\eta_\ell\Phi_\ell)^{(m-\ell)} (w) = G^{(m-\ell)} (w).\) Thus, from \eqref{L1344.1} and \eqref{L1344.2} we finally obtain \eqref{L1344.B}.
\end{proof}

\begin{lemma}\label{L1445}
 Let \(\cD\) be a region of \(\C\) such that \(\cdisk{r}\subseteq\cD\) for some \(r\in(1,+\infty)\) and let \(F\) be a \tqqa{matrix-valued} function meromorphic in \(\cD\) such that the restriction \(\Phi\) of \(F\) onto \(\D\) belongs to \(\CqD\). In view of \rrem{R1336}, let \(G\colon\cD\to\Cqq\) and \(h\colon\cD\to\C\) be holomorphic such that \(h\) does not vanish identically in \(\cD\) and that \(F=h^\inv G\) holds true. Let \(u\in\T\) be a zero of \(h\) with multiplicity \(m>0\). Then:
 \begin{enui}
  \item\label{L1445.a} \(u\) is either a removable singularity or a simple pole of \(F\).
  \item\label{L1445.b} \(h^{(m)}(u)\neq0\) and
 \begin{equation}\label{L1445.B1}
  \mu\rk*{\set{u}}
  =\frac{-m}{2uh^{(m)} (u)}G^{(m - 1)} (u),
 \end{equation}
 where \(\mu\) is the \tRHm{} of \(\Phi\).
  \item\label{L1445.c} If there is no \(z\in\cD\) with \(G(z)=\Oqq\) and \(h(z)=0\), then \(u\) is a pole of \(F\).
  \item\label{L1445.d} \(u\) is a removable singularity of \(F\) if and only if \(G^{(m - 1)} (u)=\Oqq\) or equaivalently \(\mu(\set{u})=\Oqq\).
 \end{enui}
\end{lemma}
\begin{proof}
 Obviously \(h^{(m)}(u)\neq0\) and \(u\) is either a removable singularity or a pole of \(F \), which then is simple according to \rprop{P1309}, \ie{}, the order of the pole \(u\) of \(F\) is either \(0\) or \(1\). Thus, we can chose \(\ell=1\) in \rlem{L1344} and obtain
\begin{equation}\label{L1445.1}
 \lim_{r\to 1-0}\ek*{(ru - u) F(r u)}
 = \frac{m}{h^{(m)}  (u)}G^{(m-1)} (u).
\end{equation}
 \rprop{P1309} yields \eqref{P1309.B1}. Comparing \eqref{P1309.B1} and \eqref{L1445.1}, we get \eqref{L1445.B1}. The rest is plain.
\end{proof}

 Now we will extend the statement of \rlem{L1445} for the case of rational matrix-valued functions. For this reason we will first need some notation.

 For each \(A\in\Cqq \), let \(\det A\)\index{d@$\det A$} be  the determinant of \(A\) and let \(A^\adj\)\index{$A^\adj$} be the classical adjoint of \(A\) or classical adjugate (see, \eg{}, Horn/Johnson~\cite[p.~20]{MR832183}), so that \(AA^\adj = (\det A) \Iq \) and \(A^\adj A = (\det A) \Iq \). If \(Q\) is a \tqqa{matrix} polynomial, then \(Q^\adj \colon  \C\to\Cqq\) defined by \(Q^\adj (z) \defeq  [Q(z)]^\adj\)\index{$Q^\adj (z)$} is obviously a matrix polynomial as well.

\begin{proposition}\label{P1637}%
 Let \(P\) and \(Q\) be complex \tqqa{matrix} polynomials such that \(\det Q\) does not vanish identically and the restriction \(\Phi\) of \(PQ^\inv\) onto \(\D\) belongs to \(\CqD \). Let \(u\in\T\) be a zero of \(\det Q\) with multiplicity \(m>0\). Then \(u\) is either a removable singularity or a simple pole of \(PQ^\inv\). Furthermore, \((\det Q)^{(m)}(u)\neq0\) and
 \[
  \mu\rk*{\set{u}}
  =\frac{-m}{2u (\det Q)^{(m)} (v)}   (PQ^\adj)^{(m - 1)} (u),
 \]
 where \(\mu\) is the \tRHm{} of \(\Phi\).
\end{proposition}
\begin{proof}
 The functions \(G\defeq PQ^\adj\) and \(h\defeq\det Q\) are holomorphic in \(\C\) such that \(h\) does not vanish identically, and \(F\defeq PQ^\inv\) is meromorphic in \(\C\) and admits the representation \(F=h^\inv G\). Hence, the application of \rlem{L1445} completes the proof.
\end{proof}

\begin{proposition}\label{P1643}
 Let \(Q\) and \(R\) be complex \tqqa{matrix} polynomials such that \(\det Q\) does not vanish identically and the restriction \(\Phi\) of \(Q^\inv R\) onto \(\D\) belongs to \(\CqD \). Let \(u\in\T\) be a zero of \(\det Q\) with multiplicity \(m>0\). Then \(u\) is either a removable singularity or a simple pole of \(Q^\inv R\). Furthermore, \((\det Q)^{(m)}(u)\neq0\) and
 \[
  \mu\rk*{\set{u}}
  =\frac{-m}{2u (\det Q)^{(m)} (v)}   (Q^\adj R)^{(m - 1)} (u),
 \]
 where \(\mu\) is the \tRHm{} of \(\Phi\).
\end{proposition}
\bproof
 Apply \rprop{P1637} to \(\rk{Q^\inv R}^\tra\).
\eproof

 As usual, if \(\cM\) is a finite subset of \(\Cpq \), then the notation \symb{\sum_{A\in\cM} A}{s} should be understood as \(\Opq \) in the case that \(\cM\) is empty. In the following, we continue to use the notations \(\lebc \)\index{l@$\lebc $} and \(\kron{u}\)\index{d@$\kron{u}$} to designate the linear Lebesgue measure on \((\T,\BsaT)\) and the Dirac measure on \((\T,\BsaT)\) with unit mass at \(u\in\T\), respectively. Now we are able to derive the main result of this section.

\begin{theorem}\label{T1648}
 Let \(r\in(1,+\infty)\), let \(\cD\) be a region of \(\C\) such that \(\cdisk{r}\subseteq\cD\), and let \(F\) be a \tqqa{matrix-valued} function meromorphic in \(\cD\) such that the restriction \(\Phi\) of \(F\) onto \(\D\) belongs to \(\CqD\). In view of \rrem{R1336}, let \(G\colon\cD\to\Cqq\) and \(h\colon\cD\to\C\) be holomorphic functions such that \(h\) does not vanish identically in \(\cD\) and that \(F=h^\inv G\) holds true. Then \(\cN \defeq  \setaa{u\in\T}{h(u) = 0}\) is a finite subset of \(\T \) and the following statements hold true:
 \benui
  \item\label{T1648.a} For all \(u\in\cN\), the inequality \(h^{(m_u)}  (u)\ne 0\) holds true, where \(m_u\) is the multiplicity of \(u\) as zero of \(h\), and the matrix
  \[
   W_u
   \defeq\frac{-m_u}{2uh^{(m_u)} (u)}G^{(m_u- 1)} (u)
  \]
  is well defined and \tnnH{}, and coincides with \(\mu(\set{u})\), where \(\mu\) is the \tRHm{} of \(\Phi\).  
 \item\label{T1648.b} Let \(\Delta \colon\cD\setminus\cN\to\Cqq\) be defined by
\begin{equation}\label{T1648.D1}
 \Delta (z)
 \defeq  \sum_{u\in\cN} \frac{u+z}{u-z} W_u.
\end{equation}
 Then \(\Theta \defeq F-\Delta\) is a \tqqa{matrix-valued} function meromorphic in \(\cD\) which is holomorphic in \(\cdisk{r_0}\) for some \(r_0\in(1,r)\) and the restrictions of \(\Theta\) and \(\Delta\) onto \(\D\) both belong to \(\CqD \).
\item\label{T1648.c} The \tRHm{} \(\mu\) of \(\Phi\) admits for all \(B\in\BsaT  \) the representation
\begin{equation}\label{T1648.B1}
\mu (B) = \frac{1}{2\pi} \int_B \re  \Theta (\zeta) \lebca{\dif\zeta} +  \sum_{u\in\cN} W_u   %
\kron{u} (B).
\end{equation}
\eenui
\end{theorem}
\bproof
 Since \(h\) is a holomorphic function in \(\cD\) which does not vanish identically in \(\cD\) and since \(\T\) is a bounded subset of  the interior of \(\cD\), the set \(\cN\) is finite.
 
 \eqref{T1648.a} This follows from \rlem{L1445}.
 
 \eqref{T1648.b} Obviously, \(\Theta\) is meromorphic in \(\cD\). According to \rlem{L1445}, each \(u\in\cN\) is either a removable singularity or a sinple pole of \(F \) and \(\mu (\set{u}) = W_u\) holds true. \rprop{P1309} yields then
\begin{equation}\label{T1648.1}
 \lim_{z\to u}\ek*{(z-u) F (z)}
 = -2u W_u
\end{equation}
 for each \(u\in\cN\). Obviously, \(\Theta\) is holomorphic at all points \(z\in\T \setminus\cN\). 
 
 Let us now assume that \(u\) belongs to \(\cN\). Then \(h(u) = 0\) and there is a positive real number \(r_u\) such that \(K\defeq\diskaa{u}{r_u}\) is a subset of \(\cD\) and \(h (z)\ne 0\) for all \(z\in K \setminus \set{u}\). In particular, the restriction \(\theta\) of \(\Theta\) onto \(K \setminus \set{u}\) is holomorphic and 
\begin{equation}\label{T1648.2}
(z-u)\theta  (z)
= (z-u) F (z) + (u+z) W_u - (z-u) \sum_{\zeta\in\cN \setminus \set{u}} \frac{\zeta + z}{\zeta - z} W_\zeta 
\end{equation}
 is fulfilled for each \(z\in K \setminus \set{u}\). Consequently, \eqref{T1648.1} and \eqref{T1648.2} provide us
\[\begin{split}
 \Oqq
 &= -2u W_u + (u+u) W_u - (u-u)  \sum_{\zeta\in\cN \setminus \set{u}} \frac{\zeta + z}{\zeta - z} W_\zeta\\
 &= \lim_{z\to u}\ek*{(z-u) F (z)}+ \rk{u +\lim_{z\to u} z} W_u- \ek*{\rk{\lim_{z\to u} z} -u }  \sum_{\zeta\in\cN \setminus \set{u}} \frac{\zeta + z}{\zeta - z} W_u\\
 &= \lim_{z\to u} \ek*{ (z-u) F (z) + (u+z) W_u - (z-u) \sum_{\zeta\in\cN \setminus \set{u}} \frac{\zeta + z}{\zeta - z} W_\zeta }
 = \lim_{z\to u}\ek*{(z-u) \theta   (z)}.
\end{split}\]
 In view of Riemann's theorem on removable singularities, this implies that \(u\) is a removable singularity for \(\theta \). In particular, \(\Theta\) is holomorphic at \(u\). Thus, \(\Theta\) is holomorphic at each \(\zeta\in\T \). Taking into account \(\D \cap\cN = \emptyset\), we see then that \(\Theta\) is holomorphic at each point \(z\in\D \cup\T \). Since \(\Theta\) is meromorphic in \(\cD\) and \(\cdisk{r}\) is bounded, \(\Theta\) has only a finite number of poles in \(\cdisk{r}\setminus(\D\cup\T)\). Thus, there is an \(r_0\in(1,r)\) such that \(\Theta\) is holomorphic in \(\cdisk{r_0}\). In particular, the restriction \(\Psi\) of \(\Theta\) onto \(\D\) is holomorphic. Because of \(\D \cap\cN =\emptyset\), we get
\begin{equation}\label{T1648.3}
 \Theta (z)
 =  F (z)  -\Delta (z)
 =\Phi (z) - \sum_{u\in\cN} \frac{u + z}{u - z} W_u
\end{equation}
 for each \(z\in\D \). Because of \(\mu (\set{u}) = W_u\) for each \(u\in\cN\), we conclude that
\begin{equation}\label{T1648.4}
 \rho
 \defeq  \mu - \sum_{u\in\cN} W_u\kron{u}
\end{equation}
 fulfills \(\rho (\BsaT  ) \subseteq \Cggq \) and, hence, that \(\rho\) belongs to \(\MggqT\). Since \(\mu\) is the \tRHm{} of \(\Phi\), we have \eqref{NGZ} for each \(z\in\D \). Thus, we obtain from \eqref{T1648.3} then
\[\begin{split}
\Theta (z) &= \int_\T   \frac{\zeta + z}{\zeta - z} \mu (\dif\zeta) + \iu\im  \Phi (0) - \sum_{u\in\cN}
\rk*{ \int_\T   \frac{\zeta + z}{\zeta - z} \kron{u} (\dif\zeta)} W_u  \\
&= \int_\T   \frac{\zeta + z}{\zeta - z} \rho (\dif\zeta) + \iu\im  \Phi (0) 
\end{split}\]
 for every choice of \(z\) in \(\D \). Consequently, from \rthm{D2.2.2} we see that \(\Psi\) belongs to \(\CqD \) and that \(\rho\) is the \tRHm{} of \(\Psi\). Since the matrix \(W_u\) is \tnnH{} for all \(u\in\cN\), \rthmp{D2.2.2}{D2.2.2.b} yields in view of \eqref{T1648.D1} furthermore, that the restriction of \(\Delta\) onto \(\D\) belongs to \(\CqD\) as well.

 \eqref{T1648.c} Applying \rlem{M1.2} shows then that 
\(
 \rho (B)
 =\frac{1}{2\pi} \int_B \re  \Theta (\zeta) \lebca{\dif\zeta}
\)
 holds true for each \(B\in\BsaT  \). Thus, from \eqref{T1648.4}, for each \(B\in\BsaT  \), we get \eqref{T1648.B1}.
\eproof

 A closer look at \rthm{T1648} and its proof shows that the \tRHm{s} \(\rho\) and \(\sum_{u\in\cN} W_u\kron{u}\) of \(\Psi\) and the restriction of \(\Delta\) onto \(\D\), respectively, are exactly the absolutely continuous and singular part in the Lebesgue decomposition of the \tRHm{} of \(\Phi\) with respect to \(\lebc \). In particular, the singular part is a discrete measure which is concentrated on a finite number of points from \(\T\) and there is no nontrivial singular continuous part. The absolutely continuous part with respect to \(\lebc \) possesses a continuous Radon-Nikodym density with respect to \(\lebc \).

\begin{theorem}\label{M3.2}
 Let \(P\) and \(Q\) be \tqqa{matrix} polynomials such that \(\det Q\) does not vanish identically and that the restriction \(\Phi\) of \(PQ^\inv \) onto \(\D\) belongs to \(\CqD \). Then \(\cN \defeq  \setaa{u\in\T}{\det Q (u) = 0}\) is a finite subset of \(\T \) and the following statements hold true:
\benui
 \item For all \(u\in\cN\), the inequality \((\det Q)^{(m_u)}  (u)\ne 0\) holds true, where \(m_u\) is the multiplicity of \(u\) as zero of \(\det Q\), and
  \[
  W_u
  \defeq\frac{-m_u}{2u(\det Q)^{(m_u)} (u)}   (PQ^\adj)^{(m_u - 1)} (u)
 \]
  is a well-defined and \tnnH{} matrix which coincides with \(\mu(\set{u})\), where \(\mu\) is the \tRHm{} of \(\Phi\).  
 \item  Let \(\Delta \colon\cD\setminus\cN\to\Cqq\) be defined by \eqref{T1648.D1}.
 Then \(\Theta \defeq PQ^\inv-\Delta\) is a rational \tqqa{matrix-valued} function which is holomorphic in \(\cdisk{r}\) for some \(r\in(1,+\infty)\) and the restrictions of \(\Theta\) and \(\Delta\) onto \(\D\) both belong to \(\CqD \).
\item The \tRHm{} \(\mu\) of \(\Phi\) admits the representation \eqref{T1648.B1} 
for all \(B\in\BsaT  \).
\eenui
\end{theorem}
\begin{proof}
 \rthm{M3.2} is an immediate consequence of \rthm{T1648} if one chooses \(\cD=\C\), \(h=\det Q\) and \(G=PQ^\adj\).
\end{proof}

\section{On the truncated matricial trigonometric moment problem}\label{S1029}
 A matricial version of a theorem due to G.~Herglotz shows in particular that if \(\mu\) belongs to \(\MggqT\), then it is uniquely determined by the sequence \((\Fcm{j})_{j=-\infty}^\infty\) of its Fourier coefficients given by \eqref{D2.2.9}. To recall this theorem in a version which is convenient for our further considerations, let us modify the notion of \Toe{} non-negativity.  Obviously, if  \(\kappa\in\NOinf \) and if \(\Cbk \) is a \tTnnd{} sequence, then \(C_{-j} = C_j^\ad\) for each \(j\in\mn{-\kappa}{\kappa}\). Thus, if \(\kappa\in\NOinf \), then a sequence \(\Cska \) is called \noti{\tTnnd{}} (resp.\ \noti{\tTpd{}}) if \(\Cbk \) is \tTnnd{} (resp.\ \tTpd{}), where \(C_{-j} \defeq  C_j^\ad\) for each \(j\in\mn{0}{\kappa}\). 

\begin{theorem}[G. Herglotz]\label{D2.2.1'}
 Let \(\Csinf \) be a sequence of complex \tqqa{matrices}. Then there exists a \(\mu\in\MggqT\) such that \(\Fcm{j} = C_j\) for each \(j\in\NO \) if and only if the sequence \(\Csinf \) is \tTnnd{}. In this case, the measure \(\mu\) is unique.
\end{theorem}

 In view of the fact that \(\Fcm{-j} = (\Fcm{j})^\ad\) holds true for each \(\mu \in\MggqT\) and each  \(j\in\Z\), a proof of \rthm{D2.2.1'} is given, \eg{}, in~\cite[\cthm{2.2.1}, pp.~70/71]{MR1152328}.

 In the context of  the truncated trigonometric moment problem, only a finite sequence of Fourier coefficients is prescribed: 
\begin{description}
 \item[\TMP{}:] Let \(n\in\NO \) and let \(\Csn \) be a sequence of complex \tqqa{matrices}. Describe the set \symb{\MggqTcn }{m} of all \(\mu\in\MggqT\) which fulfill \(\Fcm{j} = C_j\) for each \(j\in\Z_{0,n}\).\index{\TMP{}}
\end{description}

 The answer to the question of solvability of \tpTMP{} is as follows:

\begin{theorem}\label{D3.4.2}
 Let  \(n\in\NO \) and let \(\Csn \) be a sequence of complex \tqqa{matrices}. Then \(\MggqTcn \) is non-empty if and only if the sequence \(\Csn \) is \tTnnd{}. 
\end{theorem}

 Ando~\cite{MR0290157} gave a proof of \rthm{D3.4.2} with the aid of the Naimark Dilation Theorem. An alternate proof stated in~\cite[\cthm{3.4.2}, p.~123]{MR1152328} is connected to \rthm{D3.4.1} below, which gives an answer to the following matrix extension problem:

\begin{description}
\item[\MEP{}:] Let \(n\in\NO \) and let \(\Csn \) be a sequence of complex \tqqa{matrices}. Describe the set \symb{\Tcn}{t} of all complex \tqqa{matrices} \(C_{n+1}\) for which the sequence \(\Cs{n+1} \) is \tTnnd{}.\index{\MEP{}}
\end{description}

 The description of  \(\Tcn \), we will recall here, is given by using  the notion of a matrix ball: For arbitrary choice of \(M\in\Cpq\) , \(A\in\Cpp \), and \(B\in \Cqq\), the set \symb{\gK (M; A,B)}{k} of all \(X\in \Cpq \) which admit a representation \(X=M+AKB\) with some contractive complex \tpqa{matrix} \(K\) is said to be the matrix ball with center \(M\), left semi-radius \(A\), and right semi-radius \(B\). A detailed theory of (more general) operator balls was worked out by Yu.~L.~Smul$'$jan~\cite{MR1073857} (see also~\cite[Section 1.5]{MR1152328} for the matrix case). To give a parametrization of \(\Tcn \) with the aid of matrix balls, we introduce some further notations. For each \(A\in\Cpq\), let \(A^\mpi\)\index{$A^\mpi$} be the Moore-Penrose inverse of \(A\). By definition, \(A^\mpi\) is the unique matrix from \(\Cqp\) which satisfies the four equations
 \begin{align*}
  AA^\mpi A&=A,&
  A^\mpi AA^\mpi&=A^\mpi,&
  (AA^\mpi)^\ad&=AA^\mpi,&
  &\text{and}&
  (A^\mpi A)^\ad&=A^\mpi A.
 \end{align*}

 Let \(\kappa\in\NOinf \) and let \(\Cska \) be a sequence of complex \tqqa{matrices}. For every \(j\in\mn{0}{\kappa}\), let \(C_{-j} \defeq  C_j^\ad\). Furthermore, for each \(n\in \mn{0}{\kappa}\), let\index{t@$\Tu{n}$}\index{y@$\Yu{n}$}\index{z@$\Zu{n}$}
\begin{align}\label{NGL}
\Tu{n}&\defeq\ek{C_{j-k}}_{j,k=0}^n,&
\Yu{n}&\defeq  \col  (C_j)_{j=1}^n,&
&\text{and}&
\Zu{n}&\defeq\ek{C_n, C_{n-1},\dotsc, C_1}.
\end{align}
 Let
\begin{align}\label{MLR1}
 \Mu{1}&\defeq  \Oqq,& \Lu{1}&\defeq  C_0,&\tand{}\Ru{1}&\defeq  C_0. 
\end{align}
 If \(\kappa \ge 1\), then, for each \(n\in\mn{1}{\kappa}\), let\index{m@$\Mu{n}$}\index{l@$\Lu{n}$}\index{r@$\Ru{n}$}
\begin{align}\label{MLR}
 \Mu{n+1}&\defeq  \Zu{n}\Tu{n-1}^\dagger  \Yu{n},&\Lu{n+1}&\defeq  C_0 - \Zu{n} \Tu{n-1}^\dagger  \Zu{n}^\ad,&\tand{}\Ru{n+1}&\defeq  C_0 - \Yu{n}^\ad \Tu{n-1}^\dagger  \Yu{n}.
\end{align}
 In order to formulate an answer to \tpMEP{}, we observe, that, if \(\Cska \) is \tTnnd{}, then, for each \(n\in\mn{0}{\kappa}\), the matrices \(\Lu{n+1}\) and \(\Ru{n+1}\) are both \tnnH{} (see, \eg{},~\cite[\crem{3.4.1}, p.~122]{MR1152328}).

\begin{theorem}\label{D3.4.1}
 Let \(n\in\NO \) and let \(\Csn \) be a sequence of complex \tqqa{matrices}. Then \(\Tcn \ne \emptyset\) if and only if the sequence \(\Csn \) is \tTnnd{}. In this case, \(\Tcn  =\gK (\Mu{n+1};\sqrt{\Lu{n+1}}, \sqrt{\Ru{n+1}}).\)
\end{theorem}
 A proof of \rthm{D3.4.1} is given in~\cite[Part~I, \cthm{1}]{MR885621I_III_V}, (see also ~\cite[\cthmss{3.4.1}{3.4.2}, pp.~122/123]{MR1152328}).

 Observe that the parameters \(\Mu{n+1}\), \(\Lu{n+1}\), and \(\Ru{n+1}\) of the matrix ball stated in \rthm{D3.4.1} admit a stochastic interpretation (see~\cite[Part~I]{MR885621I_III_V}).

\bleml{L1644}
 Let \(n\in\N\) and let \(\mu\in\MggqTcn\), where \(\Csn \) is a \tTnnd{} sequence of complex \tqqa{matrices}. If \(\rank\Tu{n}\leq n\), then there exists a subset \(\cN\) of \(\T\) with at most \(nq\) elements such that \(\mu(\T\setminus\cN)=\Oqq\).
\elem
\bproof
 Let \(\mu=\matauo{\mu_{jk}}{j,k=1}{q}\) and denote by \(\stb{1},\stb{2},\dotsc,\stb{q}\) the canonical basis of \(\Cq\). We consider an arbitrary \(\ell\in\mn{1}{q}\). Then \(\Tu{n}^{(\ell)}\defeq\matauo{\Fc{\mu_{\ell\ell}}{j-k}}{j,k=0}{n}\) admits the representation
 \[
  \Tu{n}^{(\ell)}
  =\ek*{\diag_{n+1}(\stb{\ell})}^\ad\Tu{n}\ek*{\diag_{n+1}(\stb{\ell})}
 \]
 with the block diagonal matrix \(\diag_{n+1}(\stb{\ell})\in\Coo{(n+1)q}{(n+1)}\) with diagonal blocks \(\stb{\ell}\). Consequently,
 \[
  \rank\Tu{n}^{(\ell)}
  \leq\rank\Tu{n}
  \leq n.
 \]
 Hence, there exists a vector \(v^{(\ell)}\in\Co{n+1}\setminus\set{\Ouu{(n+1)}{1}}\) and \(\Tu{n}^{(\ell)}v^{(\ell)}=\Ouu{(n+1)}{1}\). With \(v^{(\ell)}=\col\seq{v^{(\ell)}_j}{j}{0}{n}\), then
 \[
  0
  =\rk{v^{(\ell)}}^\ad\Tu{n}^{(\ell)}v^{(\ell)}
  =\int_\T\abs*{\sum_{j=0}^nv^{(\ell)}_j\zeta^j}^2\mu_{\ell\ell}(\dif \zeta)
 \]
 follows. Since \(\ell\in\mn{1}{q}\) was arbitrarily chosen, we obtain \(\tr\mu(\T\setminus\cN)=\OM\), where \(\cN\) consists of all  modulus \(1\) roots of the polynomial \(\prod_{\ell=1}^q\sum_{j=0}^nv^{(\ell)}_j\zeta^j\), which is of degree at most \(nq\). Thus, by observing that \(\mu\) is absolutely continuous with respect to \(\tr\mu\), the proof is complete.
\eproof

\section{Central \tnnH{} measures}\label{S1031}
 In this section, we study so-called central \tnnH{} measures.

 Let \(\kappa\in\Ninf\) and let \(\Cska\) be a sequence of complex \tqqa{matrices}. If \(k\in\mn{1}{\kappa}\) is such that \(C_j=\Mu{j}\) for all \(j\in\mn{k}{\kappa}\), where \(\Mu{j}\) is given by \eqref{MLR1} and \eqref{MLR}, then \(\Cska\) is called \notion{\tco{k}}{central!sequence!of order $k$}. If in the case \(\kappa\geq2\) the sequence \(\Cska\) is additionally not \tco{k-1}, then \(\Cska\) is called \notion{\tcmo{k}}{central!sequence!of minimal order $k$}. If there exists a number \(\ell\in\mn{1}{\kappa}\) such that \(\Cska\) is \tco{\ell}, then \(\Cska\) is simply called \notion{\tc{}}{central!sequence}.

 Let \(n\in\NO\) and let \(\Csn\) be a sequence of complex \tqqa{matrices}. Let the sequence \((C_j)_{j=n+1}^\infty\) be recursively defined by \(C_j\defeq\Mu{j}\), where \(\Mu{j}\) is given by \eqref{MLR}. Then \(\Csinf\) is called the \notion{\tcsc{\(\Csn\)}}{central!sequence!corresponding to $\Csn$}.

\begin{remark}\label{NR1}
 Let  \(n\in\NO \) and let \(\Csn \) be a \tTnnd{} sequence of complex \tqqa{matrices}. According to \rthm{D3.4.1}, then the \tcsc{\(\Csn\)} is \tTnnd{} as well.
\end{remark}

Observe that the elements of central \tTnnd{} sequences fulfill special recursion formulas (see~\cite[Part~V, \cthm{32}, p.~303]{MR885621I_III_V} or~\cite[\cthm{3.4.3}, p.~124]{MR1152328}). Furthermore, if \(n\in\NO \) and if  \(\Csn \) is a \tTpd{} sequence of complex \tqqa{matrices}, then the \tcsc{\(\Csn \)} is \tTpd{} (see~\cite[\cthm{3.4.1(b)}]{MR1152328}).

 A \tnnH{} measure \(\mu\) belonging to \(\MggqT\) is said to be \notion{\tc}{central!measure} if \((\Fcm{j})_{j=0}^\infty\) is \tc{}. If \(k\in\N\) is such that \((\Fcm{j})_{j=0}^\infty\) is \tc{} of (minimal) order \(k\), then \(\mu\) is called \notion{\tc{} of (minimal) order \(k\)}{central!measure!of order $k$}\index{central!measure!of minimal order $k$}.
 
\begin{remark}\label{ZM}
 Let \(n\in\NO \), let  \(\Csn \) be a \tTnnd{} sequence of complex \tqqa{matrices} and let  \(\Csinf \) be the \tcsc{\(\Csn \)}. According to \rthm{D2.2.1'}, there is a unique \tnnH{} measure \(\mu\) belonging to \(\MggqT\) such that its Fourier coefficients fulfill \(\Fcm{j} = C_j\) for each \(j\in\NO \). This \tnnH{} \tqqa{measure} \(\mu\) is called the \notion{\tcmc{\(\Csn \)}}{central!measure!for $\Csn$}.
\end{remark}

\bpropl{P0958}
 Let \(n\in\N\) and let \(\Csn \) be a \tTnnd{} sequence of complex \tqqa{matrices}. Suppose \(\rank\Tu{n}=\rank\Tu{n-1}\). Then there exists a finite subset \(\cN\) of \(\T\) such that the central measure \(\muc\) corresponding to  \(\Csn \) fulfills \(\mu(\T\setminus\cN)=\Oqq\).
\eprop
\bproof
 We have \(\muc\in\MggqTcinf\) where \(\Csinf \) is the central \tTnnd{} sequence corresponding to  \(\Csn \). According to~\zitaa{MR3014198}{\cprop{2.26}}, we get \(\Lu{\ell+1}=\OM\) for all \(\ell\in\minf{n}\). In view of~\zitaa{MR3014198}{\clem{2.25}}, then \(\rank\Tu{\ell}=\rank\Tu{n-1}\) follows for all \(\ell\in\minf{n}\). In particular, \(\rank\Tu{nq}=\rank\Tu{n-1}\leq nq\). Since \(\muc\) belongs to \(\MggqTc{nq}\), the application of \rlem{L1644} completes the proof.
\eproof

 If \(n\in\N\) and if \(\Csn \) is a \tTpd{} sequence of complex \tqqa{matrices}, then the \tcmc{\(\Csn \)} is the unique measure in \(\MggqTcn \) with maximal entropy (see~\cite[Part~II, \cthm{10}]{MR885621I_III_V}).

\begin{remark}\label{R6-11R}
 Let \(\Csinf \) be a \tTnnd{} sequence which is a \tco{0}. Then it is readily checked that \(C_k = \Oqq\) for each \(k\in\N\) and that the \tc{} measure \(\mu\) corresponding to \(\Cs{0}\) admits the representation \(\mu = \frac{1}{2\pi} C_0\lebc \), where \(\lebc \) is the linear Lebesgue measure defined on \(\BsaT  \).
\end{remark}

 Now we describe the \tcmc{a finite \tTpd{} sequence of complex \tqqa{matrices}}.
\begin{theorem}\label{FK16-N}
 Let  \(n\in\NO \) and let \(\Csn \) be a \tTpd{} sequence of complex \tqqa{matrices}.  Let \(\Tu{n}^\inv  =\matauo{\tau_{jk}^{[n]}}{j,k=0}{n} \) be the \tqqa{block} representation of \(\Tu{n}^\inv \), and let the matrix polynomials \(A_n \colon \C\to\Cqq \) and \(B_n \colon\C\to\Cqq \) be given by 
\begin{align}\label{FK16-N.V1}
A_n (z)&\defeq  \sum_{j=0}^n \tau_{j0}^{[n]} z^j&
\tand{}
B_n (z)&\defeq  \sum_{j=0}^n \tau_{n,n-j}^{[n]} z^j.
\end{align}
Then \(\det A_n (z) \ne 0 \) and \(\det B_n (z)\ne 0\) hold true for each \(z\in\D \cup\T \) and the central measure \(\mu\) for \(\Csn \) admits the representations
\begin{equation}\label{NN1}
\mu (B) =\frac{1}{2\pi} \int_B [A_n (\zeta)]^{-\ast} A_n (0) [A_n (\zeta)]^\inv  \lebca{\dif\zeta}
\end{equation}
and 
\begin{equation}\label{NN2}
\mu (B) =\frac{1}{2\pi} \int_B [B_n (\zeta)]^\inv  B_n (0) [B_n (\zeta)]^{-\ast} \lebca{\dif\zeta}
\end{equation}
for each \(B\in\BsaT  \), where \(\lebc \) is the linear Lebesgue measure defined on \(\BsaT  \).
\end{theorem}

 The fact that \(\det A_n(z)\neq0\) or \(\det B_n(z)\neq0\) for \(z\in\D\cup\T\) can be proved in various ways (see \eg{}\ Ellis/Gohberg~\zitaa{MR1942683}{\csec{4.4}} or Delsarte/Genin/Kamp~\zitaa{MR0481886}{\cthm{6}}, and~\zitaa{MR1152328}{\cprop{3.6.3}, p.~336}, where the connection to the truncated matricial trigonometric moment problem is used.

 The representations \eqref{NN1} and \eqref{NN2} are proved in~\cite[Part~III, \cthm{16}, \crem{18}, pp.~332/333]{MR885621I_III_V}.

 The measure given via \eqref{NN1} was studied in a different framework by Delsarte/Genin/Kamp~\zita{MR0481886}. These authors considered a \tnnH{} measure \(\mu\in\MggqT\) with \tTpd{} sequence \((\Fcm{j})_{j=0}^\infty\) of Fourier coefficients. Then it was shown in~\zitaa{MR0481886}{\cthm{9}} that, for each \(n\in\NO\), the measure constructed via \eqref{NN1} from the \tTpd{} sequence \((\Fcm{j})_{j=0}^n\) is a solution of the truncated trigonometric moment problem associated with the sequence \((\Fcm{j})_{j=0}^n\). The main topic of~\zita{MR0481886} is to study left and right orthonormal systems of \tqqa{matrix} polynomials associated with the measure \(\mu\). It is shown in~\zita{MR0481886} that these polynomials are intimately connected with the polynomials \(A_n\) and \(B_n\) which were defined in \rthm{FK16-N}.
 
\bpropl{P1207}
 Let \(P\) be a complex \tqqa{matrix} polynomial of degree \(n\) such that \(P(0)\) is \tpH{} and \(\det P(z)\neq0\) for all \(z\in\D\cup\T\). Let \(g\colon\T\to\Cqq\) be defined by \(g(\zeta)\defeq\ek{P(\zeta)}^\invad\ek{P(0)}\ek{P(\zeta)}^\inv\). Then \(\mu\colon\BsaT\to\Cqq\) defined by
 \(
  \mu(B)
  \defeq\frac{1}{2\pi}\int_Bg(\zeta)\lebca{\dif\zeta}
 \)
 belongs to \(\MggqT\) and is central of order \(n+1\).
\eprop
\bproof
 Obviously, \(\mu\) belongs to \(\MggqT\). Let \(\Cbinf\) be the \tFcc{\(\mu\)}. According to~\cite[\clem{2}]{MR1200154}, then \(\Tu{n}\) is \tpH{}, \ie{}\ the sequence \(\Csn\) is \tTpd{}, and \(P\) coincides with the matrix polynomial \(A_n\) given in \eqref{FK16-N.V1}. In view of \rthm{FK16-N}, thus \(\mu\) is the \tcmc{\(\Csn\)}. In particular, \(\Csinf\) is the \tcsc{\(\Csn\)} and  therefore \(\Csinf\) is \tco{n+1}. Hence, \(\mu\) is \tco{n+1}.
\eproof

Using~\cite[\clem{3}]{MR1200154} instead of~\cite[\clem{2}]{MR1200154}, one can analogously prove the following dual result:
\bpropl{P1242}
 Let \(Q\) be a complex \tqqa{matrix} polynomial of degree \(n\) such that \(Q(0)\) is \tpH{} and \(\det Q(z)\neq0\) for all \(z\in\D\cup\T\). Let \(h\colon\T\to\Cqq\) be defined by \(h(\zeta)\defeq\ek{Q(\zeta)}^\inv\ek{Q(0)}\ek{Q(\zeta)}^\invad\). Then \(\mu\colon\BsaT\to\Cqq\) defined by
 \(
  \mu(B)
  \defeq\frac{1}{2\pi}\int_Bh(\zeta)\lebca{\dif\zeta}
 \)
 belongs to \(\MggqT\) and is central of order \(n+1\).
\eprop

\bpropl{P1216}
 Let \(n\in\NO\) and let \(\mu\in\MggqT\) be central of order \(n+1\) with Fourier coefficients \(\Cbinf\) such that the sequence \(\Csn\) is \tTpd{}. Then the matrix polynomials \(A_n\) and \(B_n\) given by \eqref{FK16-N.V1} fulfill \(\det A_n(z)\neq0\) and \(\det B_n(z)\neq0\) for all \(z\in\D\cup\T\), and \(\mu\) admits the representations \eqref{NN1} and \eqref{NN2} for all \(B\in\BsaT\).
\eprop
\bproof
 Since \(\mu\) is \tco{n+1}, the sequence \(\Csinf\) is \tco{n+1}. In particular, \(\Csinf\) is the \tcsc{\(\Csn\)}. Hence, \(\mu\) is the \tcmc{\(\Csn\)}. The application of \rthm{FK16-N} completes the proof.
\eproof

 In the general situation of an arbitrarily given \tTnnd{} sequence \(\Csn \) of complex \tqqa{matrices}, the \tcmc{\(\Csn  \)} can also be represented in a closed form. To do this, we will use the results on matrix-valued \Car{} functions defined on the open unit disk \(\D \) which were obtained in \rsec{S1024}.

\section{Central matrix-valued \Car{} functions}\label{S1019}
 In this section, we recall an explicit representation of the \tRHm{} of an arbitrary central matrix-valued \Car{} function.

\begin{remark}\label{NCC}
 Let \(\Csinf \) be a \tTnnd{} sequence of complex \tqqa{matrices} and let \(\Gsinf \) be given by
 \begin{align}\label{GC}
  \Gamma_0&\defeq  C_0&
  \tand{}
  \Gamma_j&\defeq  2C_j
 \end{align}
 for each \(j\in\N\). Furthermore, let \(\mu\in\MggqT\). In view of \(\Gamma_0^\ad = \Gamma_0\), \rthmss{D2.2.1'}{D2.2.2} show then that \(\mu\) belongs to \(\MggqTcinf\) if and only if \(\mu\) is the \tRHm{} of the \tqqa{\Car{}} function \(\Phi \colon\D \to\Cqq \) defined by 
\begin{equation}\label{DC}
\Phi (z)
\defeq  \int_\T  \frac{\zeta + z}{\zeta - z} \mu (\dif\zeta).
\end{equation}
\end{remark}

 The well-studied matricial version of the classical \Car{} interpolation problem consists of the following:
\begin{description}
 \item[\CIP{}:] Let \(\kappa\in\NOinf \) and let \(\Gska \) be a sequence of complex \tqqa{matrices}. Describe the set \(\CqDGka \)\index{c@$\CqDGka $} of all \(\Phi \in\CqD \) such that \(\frac{1}{j!} \Phi^{(j)} (0) = \Gamma_j\) holds true for each \(j\in\mn{0}{\kappa}\).\index{\CIP{}}
\end{description}

 In order to formulate a criterion for the solvability of \tpCIP{}, we recall the notion of a \Car{} sequence. If \(\kappa \in\NOinf \), then a sequence \(\Gska \) is called a \notion{\tqqCs{}}{\Car{}!sequence} if, for each \(n\in\mn{0}{\kappa}\), the matrix \(\re  \Sn \) is \tnnH{}, where \(\Sn \) is given by\index{s@$\Sn$}
\begin{equation}\label{N11}
 \Sn 
 \defeq
 \begin{bmatrix}
\Gamma_0 & 0                  & \hdots & 0 & 0\\
\Gamma_1 & \Gamma_0  & \hdots & 0 & 0\\
\vdots          & \vdots            &          &\vdots & \vdots\\
\Gamma_{n-1} &\Gamma_{n-2}  & \hdots & \Gamma_0 & 0\\
\Gamma_{n} &\Gamma_{n-1}  & \hdots & \Gamma_1 & \Gamma_0\\
\end{bmatrix}.
\end{equation}

\begin{theorem}\label{NLB}
 Let  \(\kappa\in\NOinf \) and let \(\Gska \) be a sequence of complex \tqqa{matrices}.  Then \(\CqDGka  \ne\emptyset \) if and only if \(\Gska \) is a \tqqCs{}.
\end{theorem}
 In the case \(\kappa =\infty\), \rthm{NLB} is a consequence of \rthmss{D2.2.2}{D2.2.1'}. In the case \(\kappa\in\NO \), a proof of \rthm{NLB} can be found, \eg{}, in~\cite[Part~I, \csec{4}]{MR885621I_III_V}.
 
\begin{cor}\label{C1343}
 Let \(\Gsinf\) be a sequence of complex \tqqa{matrices}. Then \(\Phi\colon\D\to\Cqq\) defined by
 \begin{equation}\label{Tsr}
  \Phi(z)
  =\sum_{j=0}^\infty z^j\Gamma_j
 \end{equation}
 belongs to \(\CqD\) if and only if \(\Gsinf\) is a \tqqCs{}.
\end{cor}
\begin{proof}
  Apply \rthm{NLB}.
\end{proof}

\breml{R1331}
 If \(\kappa \in\NOinf \) and a sequence \(\Gska \) of complex \tqqa{matrices} are given, then it is readily checked that \(\Gska \) is a \tqqCs{} if and only if  the sequence \(\Cska \) defined by
\begin{align}\label{CG}
C_0&\defeq  \re  \Gamma_0&
&\text{and}&
C_j&\defeq  \frac{1}{2}\Gamma_j
\end{align}
 for each \(j\in \mn{1}{\kappa}\) is \tTnnd{}.
\erem

 Let \(\kappa\in\Ninf\), let \(\Gska\) be a sequence of complex \tqqa{matrices}, and let the sequence \(\Cska\) be given by \eqref{CG} for all \(j\in\mn{0}{\kappa}\). If \(k\in\mn{1}{\kappa}\) is such that \(\Cska\) is \tc{} of (minimal) order \(k\), then \(\Gska\) is called \notion{\tCc{} of (minimal) order \(k\)}{C-central!sequence!of order $k$}\index{C-central!sequence!of minimal order $k$}. If there exists a number \(\ell\in\mn{1}{\kappa}\) such that \(\Gska\) is \tCco{\ell}, then \(\Gska\) is simply called \notion{\tCc{}}{C-central!sequence}.

 Let \(n\in\NO\), let \(\Gsn\) be a sequence of complex \tqqa{matrices}, and let the sequence \(\Csn\) be given by \eqref{CG} for all \(j\in\mn{0}{n}\). Let the sequence \((\Gamma_j)_{j=n+1}^\infty\) be given by \(\Gamma_j\defeq2C_j\), where \(\Csinf\) is the \tcsc{\(\Csn\)}. Then \(\Gsinf\) is called the \notion{\tCcsc{\(\Gsn}\)}{C-central!sequence!corresponding to $\Gsn$}.

\begin{remark}\label{R1532}
 Let  \(n\in\NO \) and let \(\Gsn \) be a \tqqCs{}. According to \rremss{R1331}{NR1}, then the \tCcsc{\(\Gsn\)} is a \tqqCs{}.
\end{remark}

 Let \(\Phi\in\CqD\) with Taylor series representation \eqref{Tsr}.
 If \(k\in\N\) is such that \(\Gsinf\) is \tCc{} of (minimal) order \(k\), then \(\Phi\) is called \notion{\tc{} of (minimal) order \(k\)}{central!\Car{} function!of order $k$}\index{central!\Car{} function!of minimal order $k$}. If there exists a number \(\ell\in\N\) such that \(\Phi\) is \tco{\ell}, then \(\Phi\) is simply called \notion{\tc{}}{central!\Car{} function}.

\begin{remark}\label{R1526}
 Let \(n\in\NO\), let \(\Gsn\) be a \tqqCs{}, and let  \(\Gsinf \) be the \tCcsc{\(\Csn \)}. According to \rrem{R1532} and \rcor{C1343}, then \(\Phi\colon\D\to\Cqq\) given by \eqref{Tsr} belongs to \(\CqD\). This function \(\Phi\) is called the \notion{\tcCfc{\(\Gsn\)}}{central!\Car{} function!for $\Gsn$}.
\end{remark}

\begin{remark}\label{NC1}
 Let \(n\in\NO \) and let \(\Csn \) be a \tTnnd{} sequence of complex \tqqa{matrices}. Further, let \(\mu\in\MggqT\). From \rrem{NCC} one can see then that \(\mu\) is the \tcmc{\(\Csn \)} if and only if \(\Phi \colon\D \to\Cqq \) defined by \eqref{DC} is the \tcCfc{the sequence \(\Gsn \)} given by \eqref{GC} for each \(j\in\mn{0}{n}\).
\end{remark}

 Let \(n\in\NO \) and let \(\Gsn \) be a sequence of complex \tqqa{matrices} such that \(\cC_q [\D , \Gsn ] \ne \emptyset\). Then \rthmss{NLB}{D3.4.1} indicate that
\[
 \left\{ \frac{1}{(n+1)!}\Phi^{(n+1)} (0)\colon \Phi\in\cC_q [\D , \Gsn ]\right\}
 = \gK \left(2\Mu{n+1}; \sqrt{2\Lu{n+1}}, \sqrt{2\Ru{n+1}}\right ),
\]
 where \(\Csn \) is given by \eqref{CG} for all \(j\in\mn{0}{n}\) (see also~\cite[Part~I, \cthm{1}]{MR885621I_III_V}).

\begin{remark}\label{R9}
 In the case \(n=0\), \ie{}, if only one complex \tqqa{matrix} \(\Gamma_0\) with \(\re  \Gamma_0\in \Cggq \) is given, the \tcCfc{\(\Gs{0}\)} is the constant function (defined on \(\D \)) with value \(\Gamma_0\) (see~\cite[\crem{1.1}]{MR2104258}).
\end{remark}

 The first and second authors showed in~\zita{MR2104258} that in the general case the \tcCfc{a \tqqCs{} \(\Gs{n}\)} is a rational matrix-valued function and constructed explicit right and left quotient representations with the aid of concrete \tqqa{matrix} polynomials. To recall these formulas, we introduce several matrix polynomials which we use if \(\kappa\in \NOinf \) and a sequence \(\Cska \) of complex \tqqa{matrices} are given.
 
 For all \(m\in\NO \) let the matrix polynomial \(e_m\) be defined by\index{e@$e_m (z)$}
\[%
 e_m (z)
 \defeq\ek{z^0 \Iq , z^1 \Iq , z^2 \Iq ,\dotsc, z^m \Iq }.
\]%

 Let \(\Gamma_0 \defeq  \re  C_0\).  For each \(j\in\mn{1}{\kappa}\), we set \(\Gamma_j \defeq  2C_j\) and \(C_{-j} \defeq  C_j^\ad\). For each \(n\in\mn{0}{\kappa}\), let the matrices \(\Tu{n}\), \(\Yu{n}\) and \(\Sn \) be defined by \eqref{NGL} and \eqref{N11}. Furthermore, for each \(n\in\mn{0}{\kappa} \), let the matrix polynomials \(\an \) and \(\bn \) be given by\index{a@$\an(z)$}\index{b@$\bn(z)$}
\begin{align}\label{N11-1N}
\an  (z)&\defeq  \Gamma_0 + z e_{n-1} (z) \Su{n-1}^\ad \Tu{n-1}^\dagger  \Yu{n}&
\tand{}
\bn  (z)&\defeq  \Iq  - z e_{n-1} (z) \Tu{n-1}^\dagger  \Yu{n}.
\end{align}

 Now we see that central \tqqa{\Car{}} functions admit the following explicit quotient representations expressed by the given data:

\begin{theorem}[{\cite[\cthm{1.2}]{MR2104258}}]\label{FK1.2}
 Let \(n\in\N\), let \(\Gsn \) be a \tqqCs{}, and let \(\Phi\) be the \tcCfc{\(\Gsn \)}. Then the matrix polynomials \(\an \) and \(\bn \) given by \eqref{N11-1N} fulfill \(\det \bn (z)\neq0\) and \(\Phi  (z) = \an  (z) [\bn  (z)]^\inv \) for all \(z\in\D \).
\end{theorem}

 Observe that further quotient representations of \(\Phi\) are given in~\cite[\cthmss{1.7}{2.3} and \cprop{4.7}]{MR2104258}. 

Obviously, the set
\[%
\Nn  \defeq  \setaa{ v\in\T}{\det \bna{v} = 0}
\]%
\index{n@$\Nn$}is finite. For each \(v\in\Nn \), let \(m_v\) be the multiplicity of \(v\) as a zero of \(\det \bn \). Then \((\det \bn )^{(m_v)} (v) \ne 0\) for each \(v\in\Nn \), so that, for each \(v\in\Nn \), the matrix
\begin{equation}\label{S1}
\Xnv   \defeq  \frac{-m_v}{2v (\det \bn)^{(m_v)} (v)} (\an \bn ^\adj)^{(m_v - 1)} (v)
\end{equation}
\index{x@$\Xnv$}and the matrix-valued functions \(\Dn\colon\C\setminus \Nn \to \Cqq \) given by
\[%
\Dna{z}
\defeq  \sum_{v\in \Nn } \frac{v+z}{v-z} \Xnv  
\]%
\index{d@$\Dna{z}$}and
\begin{equation}\label{S3}
\Lan
\defeq  \an \bn ^\inv  - \Dn
\end{equation}
\index{l@$\Lan$}are well defined.

 \rthm{FK1.2} shows that the \tc{} \Car{} function \(\Phi\) corresponding to a \tqqCs{} \(\Gs{n}\) is a rational matrix-valued function. Thus, combining \rthmss{FK1.2}{M3.2} yields an explicit expression for the \tRHm{} of \(\Phi\).

\begin{theorem}\label{C1}
 Let \(n\in\N\) and let \(\Gsn \) be a \tqqCs{}. Then the \tRHm{}  \(\mu\) of the \tcCfc{\(\Gsn \)} admits the representation
 \begin{equation}\label{C1.B1}
  \mu(B)
  = \frac{1}{2\pi} \int_B \re  \Lana{\zeta} \lebca{\dif\zeta} +  \sum_{v\in\Nn}\Xnv\kron{v} (B)
 \end{equation}
 for all \(B\in\BsaT  \), where \(\Lan\) is given via \eqref{S3} and where \(\lebc\) is the linear Lebesgue measure defined on \(\BsaT\).
\end{theorem}
\begin{proof}
 Use \rthmss{FK1.2}{M3.2}.
\end{proof}

 Now we reformulate \rthm{C1} in the language of central measures.

\begin{theorem}\label{C2}
 Let \(n\in\N\) and let \(\Csn \) be a \tTnnd{} sequence of complex \tqqa{matrices}. Then the \tc{} measure \(\mu\) for \(\Csn \) admits the representation \eqref{C1.B1} for all \(B\in\BsaT  \).
\end{theorem}
\begin{proof}
 In view of \(\re  C_0 = C_0\), the assertion follows immediately from \rremss{NC1}{NCC} and \rthm{C1}.
\end{proof}

 The following examples show in particular that central measures need neither be continuous with respect to the Lebesgue measure nor be discrete measures.
\begin{exa}[cf.~\rrem{R6-11R}]\label{E1131}
 The sequence \(\Csinf \) given by \(C_0\defeq1\) and \(C_j\defeq0\) for all \(j\in\N\) is obviously \sTnnd{}. Since \(\Mu{1}=0=C_1\) and
\(
 M_{k+1}
 =Z_kT_{k-1}^\dagger Y_k
 =\Ouu{1}{k}\cdot T_{k-1}^\dagger\cdot\Ouu{k}{1}
 =0
 =C_{k+1}
\)
for all \(k\in\N\), it is the \tcsc{\(\Cs{0}\)} and it is \tco{0}. It is readily seen that \(\frac{1}{2\pi}\lebc \) is the \tcmc{\(\Cs{0}\)} and that \(\Phi \colon\D\to\C\) defined by \(\Phi (z)=1\) is the \tcCfc{\(\Gs{0}\)}, where \(\Gamma_0\defeq1\).
\end{exa}

\begin{exa}\label{E1533}
 The sequence \(\Csinf \) given by \(C_j\defeq1\) is obviously \sTnnd{}. Since \(C_1\neq0=\Mu{1}\) and
\(
 M_{k+1}
 =Z_kT_{k-1}^\dagger Y_k
 =\mathbf{1}_k^*\rk{k^{-2}\mathbf{1}_k\mathbf{1}_k^*}\mathbf{1}_k
 =1
 =C_{k+1}
\)
 for all \(k\in\N\), where \(\mathbf{1}_k\defeq\col(1)_{j=1}^k\), it is the \tcsc{\(\Cs{1}\)} and it is \tco{1}. It is readily seen that \(\kron{1}\) is the \tcmc{\(\Cs{1}\)} and that \(\Phi\colon\D\to\C\) defined by \(\Phi(z)=(1+z)/(1-z)\) is the \tcCfc{\(\Gs{1}\)}, where \(\Gamma_0\defeq1\) and \(\Gamma_1\defeq2\).
\end{exa}

\begin{remark}\label{R1432}
 Let \(\kappa\in\NOinf\) and let \(\Cska\) and \((D_j)_{j=0}^\kappa\) be \tTnnd{} sequences of complex \tqqa{matrices} and complex \tppa{matrices}, respectively. Then the sequence \(\diag[C_j,D_j]_{j=0}^\kappa\) is \tTnnd{}.
\end{remark}

\begin{remark}\label{R1420}
 Let \(\kappa\in\Ninf\) and \(k,\ell\in\mn{1}{\kappa}\). Let \(\Cska\) be a sequence of complex \tqqa{matrices} \tco{k} and let \((D_j)_{j=0}^\kappa\) be a sequences of complex \tppa{matrices} \tco{\ell}. Then the sequence \(\diag[C_j,D_j]_{j=0}^\kappa\) is \tco{\max\set{k,\ell}}.
\end{remark}

\begin{exa}\label{E1554}
 In view of \rexass{E1131}{E1533}, one can easily see from \rremss{R1432}{R1420} that the sequence \(\Csinf \) given by \(C_0\defeq\Iu{2}\) and \(C_j\defeq\bigl[\begin{smallmatrix}0&0\\0&1\end{smallmatrix}\bigr]\) for all \(j\in\N\) is \sTnnd{}, \tco{1} and, thus, it coincides with the \tcsc{\(\Cs{1}\)}. It is readily seen that \(\bigl[\begin{smallmatrix}\frac{1}{2\pi}\lebc &0\\0&\kron{1}\end{smallmatrix}\bigr]\) is the \tcmc{\(\Cs{1}\)} and that \(\Phi \colon\D\to\Coo{2}{2}\) defined by \(\Phi (z)=\bigl[\begin{smallmatrix}1&0\\0&(1+z)/(1-z)\end{smallmatrix}\bigr]\) is the \tcCfc{\(\Gs{1}\)}, where \(\Gamma_0\defeq\Iu{2}\) and \(\Gamma_1\defeq\bigl[\begin{smallmatrix}0&0\\0&2\end{smallmatrix}\bigr]\).
\end{exa}

\begin{remark}\label{R1438}
 Let \(\kappa\in\NOinf\), let \(\Cska\) be a \tTnnd{} sequence of complex \tqqa{matrices} and let \(U\) be a unitary \tqqa{matrix}.  Then, formula~\eqref{L1348.0} below shows that the sequence \((U^\ad C_jU)_{j=0}^\kappa\) is \tTnnd{}.
\end{remark}

\begin{exa}\label{E1159}
 Let the sequence \(\Csinf \) be given by \(C_0\defeq\Iu{2}\) and \(C_j\defeq\frac{1}{4}\smat{1&\sqrt{3}\\\sqrt{3}&3}\) for all \(j\in\N\). With the unitary matrix \(U\defeq\frac{1}{2}\smat{\sqrt{3}&-1\\1&\sqrt{3}}\) we have \(C_0=U^\ad\Iu{2}U\) and  \(C_j=U^\ad\tmat{0&0\\0&1}U\) for all \(j\in\N\). In view of \rexa{E1554}, one can then easily see from \rrem{R1438} and \rlemp{L1348}{L1348.c} that the sequence \(\Csinf \) is \sTnnd{}, \tco{1}, and thus it coincides with the \tcsc{\(\Cs{1}\)}. Furthermore, \(\frac{1}{4}\smat{\frac{3}{2\pi}\lebc +\kron{1}&-\sqrt{3}(\frac{1}{2\pi}\lebc -\kron{1})\\-\sqrt{3}(\frac{1}{2\pi}\lebc -\kron{1})&\frac{1}{2\pi}\lebc +3\kron{1}}\) is the \tcmc{\(\Cs{1}\)} and \(\Phi\colon\D\to\Coo{2}{2}\) defined by \(\Phi(z)=\frac{1}{4}\smat{3+\frac{1+z}{1-z}&-\sqrt{3}(1-\frac{1+z}{1-z})\\-\sqrt{3}(1-\frac{1+z}{1-z})&1 +3\frac{1+z}{1-z}}\) is the \tcCfc{\(\Gs{1}\)}, where \(\Gamma_0\defeq\Iu{2}\) and \(\Gamma_1\defeq\frac{1}{2}\smat{1&\sqrt{3}\\\sqrt{3}&3}\).
\end{exa}

\section{The non-stochastic spectral measure of an autoregressive stationary sequence}\label{S1038}       
 Let \(\cH\) be a complex Hilbert space with inner product \(\langle.,.\rangle \). For every choice of \(g =  \col  (g^{(j)})_{j=1}^q\) and \(h = \col  (h^{(j)})_{j=1}^q\) in \(\cH^q\), the \noti{Gramian} \((g,h)\) of the ordered pair \([g,h]\) is defined by \((g,h) =\matauo{\langle  g^{(j)}, h^{(k)}}{j,k=1}{q}\)\index{$(g,h)$}. A sequence \((g_m)_{m=-\infty}^\infty\) of vectors belonging  to \(\cH^q\) is said to be stationary\index{stationary sequence} (in \(\cH^q\)), if, for every choice of \(m\) and \(n\) in \(\Z\), the Gramian \((g_m, g_n)\) only  depends on the difference \(m-n\): \((g_m, g_n) = (g_{m-n}, g_0)\). It is well known that the covariance sequence \((C_m)_{m=-\infty}^\infty\), of an arbitrary stationary sequence \((g_m)_{m=-\infty}^\infty\), given by \(C_m \defeq  (g_m, g_0)\) for each \(m\in\Z\), is \tTnnd{}, \ie{}, that, for each \(m\in\NO \), the block \Toe{} matrix \(\Tu{m} \defeq\matauo{C_{j-k}}{j,k=0}{m}\) is \tnnH{}. According to a matricial version of a famous theorem due to G.~Herglotz (see \rthm{D2.2.1'} above), there exists one  and only one \tnnH{} \tqqa{measure} \(\mu\) defined on the set \(\BsaT  \) of all Borel subsets of the unit circle \(\T  \defeq  \setaa{ \zeta\in\C}{\abs{\zeta} = 1}\) of the complex plane \(\C\)  such that, for each \(j\in\Z\), the \(j\)\nobreakdash-th \tFc{} of \(\mu\) coincides with the matrix \(C_j\). Then \(\mu\) is called the non-stochastic spectral measure of \((g_j)_{j=-\infty}^\infty\)\index{spectral measure!non-stochastic}. A stationary sequence \((g_j)_{j=-\infty}^\infty\) is said to be autoregressive\index{autoregressive sequence} if there is a positive integer \(n\) such that the orthogonal projection \(\hat{g}_n\) of \(g_0\) onto the matrix linear subspace generated by \((g_{-j})_{j=1}^n\) coincides with the  orthogonal projection \(\hat{g}\) of \(g_0\) onto the closed matrix linear subspace generated by \((g_{-j})_{j=1}^\infty\): \(\hat{g}_n = \hat{g}\). If \(\hat{g}\ne 0\), then the smallest positive integer \(n\) with \(\hat{g}_n = \hat{g}\) is called the order of the autoregressive stationary sequence \((g_j)_{j=-\infty}^\infty\)\index{autoregressive sequence!order of}. If \(\hat{g} = 0\), then \((g_j)_{j=-\infty}^\infty\) is said to be autoregressive of order \(0\).

 Now we are going to give an explicit representation of the non-stochastic spectral measure of an arbitrary autoregressive stationary sequence in \(\cH^q\), where we study the general case without any regularity conditions. This representation is expressed in terms of the covariance sequence of the stationary sequence.

 As already mentioned above, the covariance sequence \(\Cbinf \) of an arbitrary stationary sequence \((g_j)_{j=-\infty}^\infty\) in \(\cH^q\) is \tTnnd{}. Observe that, conversely, if the complex Hilbert space \(\cH\) is infinite-dimensional and if an arbitrary \tTnnd{} sequence \(\Cbinf \) of complex \tqqa{matrices} is given, then a matricial version of a famous result due to A.~N.~Kolmogorov~\cite{MR0009098} shows that there exists a stationary sequence \((g_j)_{j=-\infty}^\infty\) in \(\cH^q\) with covariance sequence \(\Cbinf \) (see also~\cite[\cthm{7}]{MR1080924}).

 The interrelation between autoregressive stationary sequences and central measures is expressed by the following theorem:
\begin{theorem}[{\cite[Part~II, \cthm{9}]{MR885621I_III_V}}]\label{FK9}
 Let \(n\in\NO \) and let \((g_j)_{j=-\infty}^\infty\) be a stationary sequence (in \(\cH^q\)) with covariance sequence \(\Cbinf \) and non-stochastic spectral measure \(\mu\). Then the following statements are equivalent:
\begin{enumerate}
 \item[(i)] \((g_j)_{j=-\infty}^\infty\) is autoregressive of order \(n\).
 \item[(ii)] \(\Csinf \) is \tco{n}.
 \item[(iii)] \(\mu\) is \tco{n}.
\end{enumerate}
\end{theorem}

 Now we are able to formulate the announced representation. 
\begin{theorem}\label{C}
 Let \((g_j)_ {j=-\infty}^\infty\) be a stationary sequence in \(\cH^q\) with covariance sequence \(\Cbinf \) and let \(n\in\N\). Suppose  that \((g_j)_{j=-\infty}^\infty\) is autoregressive of order \(n\). Then \(\Lan\) given by \eqref{S3} is holomorphic at each point \(u\in\T \) and the non-stochastic spectral measure \(\mu\) of \((g_j)_{j=-\infty}^\infty\) admits the representation \eqref{C1.B1}
 for all \(B\in \BsaT  \), where \(\lebc \) is the linear Lebesgue measure defined on \(\BsaT  \), the matrix \(\Xnv \) is given by \eqref{S1}, and \(\kron{v}\) is the Dirac measure defined on \(\BsaT  \) with unit mass at \(v\).
\end{theorem}
\begin{proof}
 According to \rthm{FK9}, the sequence \(\Csinf \) is \tco{n} and \(\mu\) is \tco{n}. From the definition of the  non-stochastic spectral measure of \((g_j)_{j=-\infty}^\infty\) we know then that \(\mu\) is the \tcmc{\(\Csn \)}. Consequently, the application of \rthm{C2} completes the proof.
\end{proof}

\begin{remark}\label{ZR}
 Let \((g_j)_{j=-\infty}^\infty\) be a stationary sequence in \(\cH^q\) which is autoregressive of order \(0\). Then the non-stochastic spectral measure \(\mu\) of \((g_j)_{j=-\infty}^\infty\) is given by \(\mu = \frac{1}{2\pi} (g_0, g_0)\lebc \) (see \rthm{FK9} and \rrem{R6-11R}).
\end{remark}

\appendix
\section{Some facts from matrix theory}\label{A1346}
\breml{L1342}
 Let \(A\in\Cpq\). Further, let \(V\in\Coo{m}{p}\) and \(U\in\Coo{q}{n}\) satisfy the equations \(V^\ad V=\Ip\) and \(UU^\ad=\Iq\), respectively. Then \((VAU)^\mpi=U^\ad A^\mpi V^\ad\). 
\erem

\bleml{L1348}
 Let \(\kappa\in\NOinf\) and let \(\Cska\) be a sequence from \(\Cqq\). Let \(U\in\Cqq\) be unitary and let \(C_{j,U}\defeq U^\ad C_jU\) for \(j\in\mn{0}{\kappa}\). For \(j\in\mn{0}{\kappa}\) let \(C_{-j}\defeq C_j^\ad\) and \(C_{-j,U}\defeq C_{j,U}^\ad\).
 \benui
  \il{L1348.a} Let \(n\in\mn{0}{\kappa}\). Let \(\Tn\defeq\matauo{C_{j-k}}{j,k-0}{n}\) and \(\Tu{n,U}\defeq\matauo{C_{j-k,U}}{j,k-0}{n}\). Then
  \begin{equation}\label{L1348.0}
   \Tu{n,U}
   =\ek*{\diag_{n+1}\rk{U}}^\ad\Tn\ek*{\diag_{n+1}\rk{U}}
  \end{equation}
  and
  \begin{equation}\label{L1348.1}
   \Tu{n,U}^\mpi
   =\ek*{\diag_{n+1}\rk{U}}^\ad\Tn^\mpi\ek*{\diag_{n+1}\rk{U}}.
  \end{equation}
  \il{L1348.b} Let \(n\in\mn{0}{\kappa}\). Let \(\Yu{n}\) and \(\Zu{n}\) be given by \eqref{NGL}. Furthermore let \(\Yu{n,U}\) and \(\Zu{n,U}\) be defined by \(\Yu{n,U}\defeq  \col  (C_{j,U})_{j=1}^n\) and \(\Zu{n,U}\defeq\mat{C_{n,U},\dotsc, C_{1,U}}\). Let \(\Mu{1}\), \(\Lu{1}\), and \(\Ru{1}\) be given by \eqref{MLR1}, let \(\Mu{1,U}\defeq\Oqq\), \(\Lu{1,U}\defeq  C_{0,U}\), and let \(\Ru{1,U}\defeq  C_{0,U}\). If \(\kappa\geq1\), then, for each \(n\in\mn{1}{\kappa}\), let \(\Mu{n+1}\), \(\Lu{n+1}\), and \(\Ru{n+1}\) be given via \eqref{MLR}, let
 \begin{align*}
  \Mu{n+1,U}&\defeq  \Zu{n,U}\Tu{n-1,U}^\dagger  \Yu{n,U},&
  \Lu{n+1,U}&\defeq  C_{0,U}- \Zu{n,U} \Tu{n-1,U}^\dagger  \Zu{n,U}^\ad
 \end{align*}
 and
 \[
  \Ru{n+1,U}
  \defeq  C_{0,U}- \Yu{n,U}^\ad \Tu{n-1,U}^\dagger  \Yu{n,U}.
 \]
  For each \(n\in\mn{0}{\kappa}\) then
  \begin{align*}
 \Mu{n+1,U}&=U^\ad\Mu{n+1}U,&
 \Lu{n+1,U}&=U^\ad\Lu{n+1}U,&
 &\text{and}&
 \Ru{n+1,U}&=U^\ad\Ru{n+1}U.
 \end{align*}
  \il{L1348.c} If \(k\in\mn{2}{\kappa}\) and if \(\Cska\) be \tco{k}, then \(\seq{C_{j,U}}{j}{1}{\kappa}\) is \tco{k}.
  \il{L1348.d} If \(k\in\mn{2}{\kappa}\) and if \(\Cska\) be \tcmo{k}, then \(\seq{C_{j,U}}{j}{1}{\kappa}\) is \tcmo{k}.
 \eenui
\elem
\bproof
 Equation~\eqref{L1348.0} is obvious. Since \(U\) is unitary, the matrix \(\diag_{n+1}(U)\) is unitary as well. Thus, in view of \rrem{L1342}, formula~\eqref{L1348.1} is an immediate consequence of \eqref{L1348.0}. \rPart{L1348.a} is proved. Obviously,
 \(
  \Mu{1,U}
  =\Oqq
  =U^\ad\Mu{1}U\).
 Now, let \(n\in\mn{1}{\kappa}\). Then, using~\eqref{L1348.a} and \(\ek{\diag_n\rk{U}}\ek{\diag_n\rk{U}}^\ad=\Iu{nq}\), we get
 \[\begin{split}
  \Mu{n+1,U}
  &=\Zu{n,U}\Tu{n-1,U}^\dagger  \Yu{n,U}\\
  &=\mat{U^\ad C_nU,\dotsc,U^\ad C_1U}\ek*{\diag_n\rk{U}}^\ad\Tu{n-1}^\mpi\ek*{\diag_n\rk{U}}\ek*{\col\seq{U^\ad C_jU}{j}{1}{n}}\\
  &=U^\ad \mat{C_n,\dotsc,C_1}\Tu{n-1}^\mpi\ek*{\col\seq{C_j}{j}{1}{n}}U
  =U^\ad\Zu{n}\Tu{n-1}^\mpi\Yu{n}U
  =U^\ad\Mu{n+1}U.
 \end{split}\]
 Analogously, the remaining assertions of~\eqref{L1348.b} can be shown. The assertions stated in~\eqref{L1348.c} and~\eqref{L1348.d} are an immediate consequence of~\eqref{L1348.b}.
\eproof

\bibliography{151arxiv}
\bibliographystyle{abbrv}%

\vfill\noindent
\begin{minipage}{0.5\textwidth}
 Universit\"at Leipzig\\
Fakult\"at f\"ur Mathematik und Informatik\\
PF~10~09~20\\
D-04009~Leipzig
\end{minipage}
\begin{minipage}{0.49\textwidth}
 \begin{flushright}
  \texttt{
   fritzsche@math.uni-leipzig.de\\
   kirstein@math.uni-leipzig.de\\
   maedler@math.uni-leipzig.de
  } 
 \end{flushright}
\end{minipage}

\end{document}